\documentclass{rmmcart}
\usepackage {amsmath,amssymb,amsthm}
\usepackage{hyperref}
\usepackage{graphicx, color}
\newcommand{\hada}{\star}
\newcommand{\PP} {{\mathbb{P}}}
\newcommand{\LL}{{\mathcal{L}} }

\title {Weak Hadamard Star Configurations and Apolarity}
\author{Iman Bahmani Jafarloo}
\address{ DISMA, Dipartimento di Scienze Matematiche, Politecnico di Torino, Corso Duca degli Abruzzi 24, 10129 Turin, Italy.  \newline
	 Dipartimento di Matematica, Universit\`a degli Studi di Torino, Via Carlo Alberto 10, 10123 Turin, Italy.}
\email{iman.bahmanijafarloo@polito.it}
\thanks{Corresponding author: Iman Bahmani Jafarloo}
\thanks{The first author acknowledges that the present research has been partially supported by MIUR grant Dipartimento di Eccellenza 2018-2022 (E11G18000350001).}
\author{Gabriele Calussi}
\address{Dipartimento di Matematica e Informatica, Universit\`a di Perugia, Via Luigi Vanvitelli 1, 06123 Perugia,  Italy. \newline  Dipartimento di Matematica e Informatica, Universit\`a di Firenze, Viale Morgagni 67/A, 50134 Firenze,  Italy.
	}
\email{gabriele.calussi@gmail.com}
\date{December 30, 2018}
\begin{document}
\maketitle
\begin{abstract}
Recently, E. Carlini, M.V. Catalisano, E. Guardo, A. Van Tuyl have introduced a new construction using the Hadamard product to present star configurations of codimension $ c $ of $\PP^n$ and which they called Hadamard star configurations. In this paper, we introduce a more general type of Hadamard star configuration. Any star configuration constructed by our approach is called a weak Hadamard star configuration. We classify weak Hadamard star configurations, and in the case $ c=n $, we investigate the existence of a (weak) Hadamard star configuration apolar to the generic homogeneous polynomials of degree $ d$.
\end{abstract}
\section{Introduction}
A codimension $ c $ star configuration in $ \PP^n$ is determined by a union of linear subspaces $U_1,\ldots,U_s$ each of codimension $ c $. When $ c=n $, a star configuration forms a set of $ s $ complete intersection points in $ \PP^n $, and its defining ideal has been studied widely. For instance, in \cite{GHM}, the authors compared the symbolic and ordinary powers of the ideal of star configuration, and in \cite{CGV} an investigation was done when a generic hypersurface contains a star configuration. More recently in \cite{BC}, the authors investigated whether a star configuration of codimension $ n $ represents a given generic homogeneous form of degree $ d $ as the sum of $ d $-th power of $ s $ linear forms, see Section \ref{AHS}.

Around 2010, in \cite{CMS,CTY}, the Hadamard product of matrices was extended to Hadamard product of varieties in the study of the geometry of Boltzmann machines.
\begin{definition}[\cite{CMS}]
Given any two subvarieties $ X $ and $ Y$ of a projective space $ \PP^n$, we define their Hadamard product $ X\hada Y $ to be the closure of the image of the rational map
$$X\times Y\dashrightarrow \PP^n, (A,B)\mapsto (a_0b_0:a_1b_1:\cdots:a_nb_n).$$
For any projective variety $ X $, we may consider its Hadamard square $ X^{[2]} = X \hada X $ and its higher Hadamard powers $ X^{[k]} = X\hada X^{[k-1]} $.
\end{definition}\vspace*{-0.5cm}
Its applications can be also found in tropical geometry \cite{FOW,MS}.
Recently in \cite{BCK}, the authors studied Hadamard product of linear spaces and obtained its connections with tropical geometry. Other papers that contributed to the study of Hadamard product of varieties include \cite{BCFL1,BCFL2,CCFL}. Using Hadamard product the authors in \cite[Theorem 4.7]{BCK} constructed a new family of star configurations of codimension $ n $. Indeed, they showed that the Hadamard square of a set of points on a given line in $ \PP^n $ is a star configuration of points. Later in \cite{CCGV} it was generalized for any codimension $ c $ and called Hadamard star configuration.

In this paper, we introduced star configurations of codimension $ c $ in $ \PP^n $ which are more general than Hadamard star configurations and we called them weak Hadamard star configuration. Using the genericity property of points under the standard Cremona transformation, we generalized \cite[Theorem 3.1]{CCGV} in Theorem \ref{teowhsc} and in Theorem \ref{thm411} we extended \cite[Theorem 2.17]{CCGV}.  Moreover, two important consequences of Theorem  \ref{teowhsc} are Corollaries \ref{cor1}, \ref{cor2} which characterize weak Hadamard star configurations.

This paper is organized as follows: in the next section we review some of the standard facts on Hadamard product and introduce weak Hadamard star configuration. In Section 3, by the standard Cremona transformation we prove some lemmas and useful tools for characterization of (weak) Hadamard star configuration. In Section 4, our main results are stated and proved (see the previous paragraph). In particular, we find the relation between the star configuration constructed via our approach and Hadamard star configuration. As in \cite{BC}, Section 5 is intended to motivate our investigation of the existence of a (weak) Hadamard star configuration of codimension $ n $ apolar to a given generic homogeneous form (see Lemma \ref{proapolar} and Example \ref{nicex}).

\textbf{Acknowledgements.} The first author thanks the Dipartimento di Matematica e Informatica, Universit\`a di Perugia, Italy for its hospitality and the Dipartimento di Scienze Matematiche, Politecnico di Torino, Italy for financial support.
The second author thanks the Politecnico of Torino for its hospitality and GNSAGA of INdAM and MIUR for their partial support. We also wish to thank Enrico Carlini, Giuliana Fatabbi and Anna Lorenzini for their help. 

\section{Preliminary definition}
In this section, we recall some relevant basic facts. Let $ S=\mathbb{C}[x_0,\ldots,x_n] $ be the standard graded polynomial ring. We denote by $ S_i $ the homogeneous degree $ i $ part of $ S $. For any homogeneous ideal $ I\subset S $, we denote by $ V(I)\subset\PP^n $ the variety defined by the vanishing locus of all elements of $ I $. 
\begin{definition}
Let $ r\geq n+1 $ and let $P=\{ P_1,\ldots,P_r\}$ be set a of points in $ \PP^n $. We say that $ P $ is in \emph{general position} if there exists no hyperplane containing any subset of $ n+1 $ distinct elements in $ P $.
\end{definition}
\begin{remark}\label{Remarkgenpoints}
From the definition it follows that $ P_1,\ldots,P_r $ are in general position if and only if the matrix $ \begin{pmatrix}
P_1& \cdots& P_r\\
\end{pmatrix} ^T$ has all non zero maximal minors.
\end{remark}
\begin{definition}
Let $ H_i=V(x_i)$ for $ i=0,\ldots n $ be the coordinate hyperplanes of $ \PP^n $. Let
$$\Delta_i=\bigcup_{0\leq j_1<\ldots<j_{n-i}\leq n}H_{j_1}\cap\ldots\cap H_{j_{n-i}}.$$
\end{definition}
\begin{definition}
Let $ \mathcal{L}=\{L_1,\dots,L_r\} $ be a set of linear forms in $ S_1 $. The set $ \LL $ is  \emph{generally linear} if the following hold
\begin{enumerate}
\item[$ \bullet $] $ \{L_1,\ldots,L_t\}\subseteq \LL $ and $ t\leq n $, then $\dim(L_1\cap\ldots\cap L_t)=n-t $,
\item[$\bullet$] $ \{L_1,\ldots,L_t\}\subseteq \LL $ and $ t> n $, then $L_1\cap\ldots\cap L_t=\emptyset $.
\end{enumerate}
\end{definition}
\begin{definition}
Let $ \mathcal{L}=\{L_1,\dots,L_r\}$ be a set of generally linear forms in $ S_1 $. For any $ c\in[n]:=\{1,\ldots,n\} $, the \emph{codimension $ c $ star configuration} or simply  \emph{star configuration} defined by $ \LL $ is:
$$\mathbb{X}_c(\mathcal{L})=\bigcup_{1\leq i_1<\ldots<i_c\leq r}V(L_{i_1},\ldots,L_{i_c}).$$
\end{definition}
Here, we recall some definitions from \cite{CCGV}.
\begin{definition}
Let $ A=[a_0:\cdots:a_n]$ and $B=[b_0:\cdots:b_n] $ be two points in $ \PP^n $. If $ a_ib_i\neq 0 $ for some $ i $, the \emph{Hadamard product} $ A\hada B $ of $ A $ and $ B $, is defined as
$$A\hada B=[a_0b_0:a_1b_1:\cdots:a_nb_n].$$ 
If $ a_ib_i = 0 $ for all $ i = 0,\ldots,n $ then we say $ A\hada B $ is not defined.
\end{definition}
\begin{definition}
Let $ \mathcal{L}=\{L_1,\dots,L_r\}\subset S_1 $ be a set of linear forms. We say that $ \mathcal{L} $ is a \emph{Hadamard set} if there exists a linear form $ L=a_0x_0+\cdots+a_nx_n \in S_1$ and $ P_1,\ldots,P_r $ points of $ \PP^n $ such that $ V(L_i)=P_i\hada V(L) $ for all $ i\in[r] $.
\end{definition}
\begin{remark}\label{rem2}
Let $ H =V(a_0x_0+\cdots+a_nx_n)$ be a hyperplane in $ \PP^n $. Let $ P=[p_0:\cdots:p_n]\in\PP^n\setminus\Delta_{n-1} $. Then $ P\hada H =V\left(\dfrac{a_0x_0}{p_0}+\cdots+\dfrac{a_nx_n}{p_n}\right)$, see \cite[Lemma 2.13]{CCGV}.	
\end{remark}
\begin{definition}
Let $ \mathcal{L}=\{L_1,\dots,L_r\}$ be a Hadamard set. We say $ \mathcal{L} $ is a \emph{strong Hadamard set} if $ P_i\in V(L) $ for all $ i\in[r] $.
\end{definition}
\begin{definition}
A star configuration $ \mathbb{X}_c(\mathcal{L}) $ is called a
\begin{itemize}
	\item[{(a)}] {weak Hadamard star configuration} (WHSC) if $ \mathcal{L} $ is a Hadamard set.
	\item[{(b)}]  {Hadamard star configuration} (HSC) if $ \mathcal{L} $ is a strong Hadamard set.
\end{itemize}
\end{definition}
\section{Properties of Standard Cremona transformation}
In this section we study some properties of the Standard Cremona transformation. We denote by $ \sigma:\PP^n \dashrightarrow \PP^n $ the Standard Cremona transformation $$ \sigma\left( [p_0:\cdots:p_n]\right)=\left[ \dfrac{1}{p_0}:\cdots:\dfrac{1}{p_n}\right].$$
\begin{definition}
The points $ P_r,\ldots,P_r\in \PP^n $  are generic points if there exists a open dense subset $ U\subseteq(\PP^n)^r $ such that $ (P_1,\ldots,P_r)\in U $.
\end{definition}
\begin{lemma}\label{lemmacremona}
If $ P_1,\ldots,P_r\in \PP^n\setminus\Delta_{n-1} $ with $ P_i=[p_0(i):\cdots:p_n(i)] $, then $ \sigma(P_1),\ldots,\sigma(P_r) $ are in general position if and only if the following matrix $ M $ has all non-zero maximal minors
$$ M=\begin{bmatrix}
\dfrac{1}{p_0(1)}& \cdots& \dfrac{1}{p_n(1)}\\
\vdots&\space& \vdots\\
\dfrac{1}{p_0(r)}&\cdots&\dfrac{1}{p_n(r)}
\end{bmatrix}_.$$
\end{lemma}
\begin{proof}
Observe that $ M=\begin{bmatrix}
\sigma(P_1)&\ldots&\sigma(P_r)
\end{bmatrix}^T$. Therefore the proof follows from Remark \ref{Remarkgenpoints}.
\end{proof}
\begin{lemma}\label{genericlemma}
If $ P_1,\ldots,P_r $ are generic points in $ \PP^n $, then $ \sigma(P_1),\ldots,\sigma(P_r) $ are in general position.
\end{lemma}
\begin{proof}
In order to prove that $ \sigma(P_1),\ldots,\sigma(P_r) $ are in general position, it suffices to show that all maximal minors of the matrix $ M $ in Lemma \ref{lemmacremona} are non-zero. For any subset $\mathcal{I}=\{i_1,\ldots,i_{n+1}\}$ of $ n+1 $ distinct elements of $[r] $, we define  $\lambda_{\mathcal{I}}:= \det(M_{\mathcal{I}}) $ where
$$ M_{\mathcal{I}}=\begin{bmatrix}
\sigma(P_{i_{1}})&\ldots&\sigma(P_{i_{n+1}})
\end{bmatrix}^T.$$
Indeed,  $ \lambda_{\mathcal{I}} $ is a maximal minor of $ M $. We define the multi-homogeneous polynomial
$$ F_{\mathcal{I}}=\lambda_{\mathcal{I}} p_0(i_1)\cdots p_n(i_1)\cdots p_0(i_{n+1})\cdots p_n(i_{n+1}) $$
in the multi-graded polynomial ring $$\mathbb{C}[p_0(1),\ldots,p_n(1),\ldots,p_0(r),\ldots,p_n(r)] $$ of multi-degree $ (n,\ldots,n, 0,\ldots,0) $. Since $ F_{\mathcal{I}} $ is non-zero, it follows that $ C_{\mathcal{I}}=V(F_{\mathcal{I}}) $ is a proper closed subset in $\PP^n\times\cdots\times\PP^n$, ($ r $ times) and their union  denoted by $ C $ is still a proper closed subset. Without loss of generality and using the genericity of $ P_1,\ldots, P_r $ we can assume that $(P_1,\ldots, P_r)\notin C $. From the definition 
$$ C=V\left(\bigcap <F_{\mathcal{I}}>\right),$$
 we conclude that $ F_{\mathcal{I}}( P_1,\ldots, P_r)\neq 0 $. It follows that $ \lambda_{\mathcal{I}}\neq 0 $. Hence, all maximal minors of the matrix $ M $ are non-zeros and the proof is now completed.
\end{proof}
\begin{remark}\label{remarkgeneralnogeneric}
Let $ P_1,\ldots,P_r\in\PP^n $ be in general position. Then $\sigma(P_1),\dots,\sigma(P_r)$ do not necessarily have the same property. In fact, assume that $ H $  is a generic hyperplane in $ \PP^n $. Then $\sigma(H)$ is a hypersurface of degree $d>1$ and so there exist $P_{i_1},\dots,P_{i_{n+1}}$ in $\sigma(H)$ which are in general position. Since $\sigma(\sigma (H))=H$, we deduce that $\sigma(P_{i_1}),\dots,\sigma(P_{i_{n+1}})$ are not in general position on $ \PP^n $.
\end{remark}
\begin{lemma}\label{lemmgenpointhyp}
Let $F$ be a degree $ d $ form in $ S $ and $n\geq 2$. If $V(F)$ is an irreducible hypersurface of degree $d>1$, then for all $k\in \mathbb{N}$ there exist $P_1,\dots,P_k\in V(F)$ such that $P_1,\dots,P_k$ are in general position.
\end{lemma}
\begin{proof}
We give a proof by induction on $k$. If $k=n+1$, then $n+1$ points are in general position if and only if they span $\PP^n$.
It is clear that $V(F)$ must be a non-degenerate hypersurface, and so there are $n+1$ points on $V(F)$ which span $\PP^n$.
Now assume that $k>n+1$ and by induction there exist $ k-1 $ points $P_1,\dots,P_{k-1}$ on $V(F)$ in general position. Let $V(L_1),\dots,V(L_t)$ be the hyperplanes generated by any choice of $n$ distinct points in $\{ P_1,\dots,P_{k-1}\}$. Since $V(F)$ is irreducible and not contained in any hyperplane $ V(L_i)$, we thus have that $V(L_i)\cap V(F)$ has dimension $n-2$ for all $i$. Hence, $V(F)\not\subset (V(L_1)\cup\cdots\cup V(L_t))$ and there exists $P_k\in V(F)\setminus (V(L_1)\cup\cdots\cup V(L_t))$. It follows that the points $P_1,\dots,P_{k}$ are in general position.
\end{proof}
\begin{lemma}\label{genericlemmahyp}
Let $ H\subset\PP^n$ be a generic hyperplane and let $ P_1,\ldots,P_r $ be generic points on $ H $. Then $ \sigma(P_1),\ldots,\sigma(P_r) $ are in general position.
\end{lemma}
\begin{proof}
This follows by the same method as in the proof of Lemma \ref{genericlemma}. With out loss of generality, let $a_0x_0+\dots+a_nx_n=0$ be the equation of $ H $ with $ a_i\neq 0 $ for $ i=0,\ldots,n $.
The hyperplane $ H $ vanishes on $ P_i=[p_0(i):\cdots:p_n(i)] $ for $ i=1,\ldots,r $ if and only if
\begin{equation}\label{pni}
p_n(i)=-(a_0p_0(i)+\dots+a_{n-1}p_{n-1}(i))/{a_n}.
\end{equation}
We now apply the same argument in the proof of Lemma \ref{genericlemma}, with $p_n(i)$ replaced by 
the linear combinations of $ p_0(i),\ldots,p_{n-1}(i) $, see Equation (\ref{pni}), to obtain the multi-homogeneous polynomial $ G_{\mathcal{I}} $ in the multi-graded polynomial ring $$\mathbb{C}[p_0(1),\ldots,p_{n-1}(1),\ldots,p_0(r),\ldots,p_{n-1}(r)] .$$ 
Thus the statement holds on:
$$\mathcal{A}=\underbrace{\PP^n\times\cdots\times\PP^n}_{r \ \text{times}}\setminus\left(\bigcup V(G_{\mathcal{I}})\right).$$
Now we prove that $\mathcal{A}$ is a non-empty open subset. In fact, the set $ V(G_{\mathcal{I}}) $ is not necessarily a proper set since $ G_{\mathcal{I}} $ might be zero. We know that $\sigma(H)$ is an irreducible hypersurface of degree $d>1$, so by Lemma \ref{lemmgenpointhyp}, there are $r$ points $ Q_1,\ldots,Q_r $ in general position on $\sigma(H)$. Take $ P_1,\ldots P_r $ such that $ \sigma(P_1)=Q_1,\ldots,\sigma(P_r)=Q_r $. Since $ \sigma(P_1),\ldots,\sigma(P_r) $ are in general position in $ \PP^n $, we have that $ (P_1,\ldots,P_r)\in\mathcal{A} $, so $ \mathcal{A} $ is non-empty.
\end{proof}
\begin{remark}\label{examplegeneralnogeneric}
Note that if the points $ P_1,\ldots,P_r $ are in general position on $ H $, then it does not guarantee that $\sigma(P_1),\ldots,\sigma(P_r)$ are in general position. For example, let $H=V(x_0+2x_1+3x_3-x_4)\subset\PP^3$ and consider the points $P_1=[1:2:3:14]$, $P_2=[1:1:1:6] $, $P_3=[-1:2:-2:-3]$ and $P_4=[-1:-2:190/33:135/11]$ in general position on $ H $. We have that $ \det( \begin{bmatrix}
\sigma(P_1)&\cdots&\sigma(P_4)
\end{bmatrix}^T)=0 $ and it follows that $\sigma(P_1),\ldots,\sigma(P_4)$ are not in general position.
\end{remark}
\section{Weak Hadamard star configuration}
Our goal in this section is to find the necessary and sufficient condition for a generally linear set of linear forms to be a WHSC.
\begin{definition}
Let $ L $ be a linear form. The {\em support} of $ L $ is the set of variables appearing in $ L $ with non-zero coefficient.
\end{definition}
\begin{proposition}\label{propsamesupport}
Let $\mathcal{L}=\{L_1,\dots,L_r\}$ be a generally linear set of linear forms in $ S_1 $. The set $\mathbb{X}_c(\mathcal{L}) $ is a WHSC if and only if  $V(L_i)\cap \Delta_0=\emptyset$ for all $i\in [r]$.
\end{proposition}
\begin{proof}
If $V(L_i)\cap \Delta_0=\emptyset$ for all $i\in [r]$, then from \cite[Remark 2.10]{CCGV} we conclude that  $\mathbb{X}_c(\mathcal{L})  $ is a WHSC. Conversely, if $\mathbb{X}_c(\mathcal{L}) $ is a WHSC, then from \cite[Remark 2.10]{CCGV} it follows all the linear forms $ L_i $ have the same support. By contradiction, suppose that there exists $t\in[r]$ such that $V(L_t)\cap\Delta_0\neq \emptyset$. The fact that $ L_i $'s have the same support implies that there is at least one zero coefficient in their support. With out loss of generality, assume that the first coefficients are zero. It follows that $ L_i\in \mathbb{C}[x_1,\ldots,x_n] $ for all $i\in[r] $ which is impossible since $ \LL $ is generally linear in $ S_1 $.
\end{proof}
Let $ H_1,\ldots,H_r $ be hyperplanes of $ \PP^n $. We denote by
$$\mathbb{X}_c(H_1,\ldots,H_r)=\bigcup_{1\leq i_1<\cdots< i_c\leq r}H_{i_1}\cap\cdots\cap H_{i_c}.$$

\begin{theorem}\label{teowhsc}
Let $ H\subset\PP^n $ be a hyperplane such that $ H\cap \Delta_0=\emptyset $. Consider  $ P_1,\ldots,P_r\in \PP^n\setminus\Delta_{n-1}$ and set $ H_j=P_j\hada H$ where $ P_j=[p_0(j):\cdots:p_n(j)] $ for all $ j\in[r]$. Then $ \mathbb{X}_c(H_1,\ldots,H_r)$ is a WHSC if and only if the points $ \sigma(P_1),\ldots,\sigma(P_r) $ are in general position in $ \PP^n $.
\end{theorem}
\begin{proof}
Assume that $ H=V(a_0x_0+\cdots+a_nx_n) $  with $ a_i\neq0 $ for all $ i=0,\dots,n$.  Let $ \mathcal{L}=\{L_1,\ldots,L_r\} $ be a set of linear forms in $ S_1 $ where
$$ L_j=\dfrac{a_0x_0}{p_0(j)}+\cdots+\dfrac{a_nx_n}{p_n(j)},\quad \forall j\in[r]. $$
Set $ H_j=V(L_j) $ since $ V(L_j)=P_j\hada H $. What remains to prove is: the set $\mathcal{L}$ is generally linear if and only if $ \sigma(P_1),\ldots,\sigma(P_r) $ are in general position. Suppose that $ \mathcal{L}$ is not generally linear, i.e., there exist $ n+1 $ distinct elements in $\LL$ which are linearly dependent, say $ L_1,\ldots, L_{n+1}$. Therefore, there exist $ \lambda_j\neq 0 $ for $ j=1,\ldots,n+1 $ such that $ \sum\limits_{j=1}^{n+1}\lambda_jL_j=0 $. Hence,
\begin{flalign}\label{eq2}
\sum\limits_{j=1}^{n+1}\lambda_jL_j=&\sum_{j=1}^{n+1}\lambda_j\left( \dfrac{a_0x_0}{p_0(j)}+\cdots+\dfrac{a_nx_n}{p_n(j)}\right)\\ 
=&\sum_{i=0}^{n}\left( \dfrac{\lambda_1}{p_i(1)}+\cdots+\dfrac{\lambda_{n+1}}{p_i(n+1)}\right)a_ix_i=0.\nonumber
\end{flalign}
Since $ a_i\neq 0$ for all $i=0,\dots,n$, we get the following system:
\begin{equation}\label{eq3}
\left\lbrace \begin{array}{c}
\dfrac{\lambda_1}{p_0(1)}+\cdots+\dfrac{\lambda_{n+1}}{p_0(n+1)}=0\\
\vdots\\
\dfrac{\lambda_1}{p_n(1)}+\cdots+\dfrac{\lambda_{n+1}}{p_n(n+1)}=0
\end{array}\right..
\end{equation} 
We conclude from $ \lambda_j\neq 0 $ for $ j=1,\ldots,n+1 $ that the system has not only the trivial solution, hence that
 \begin{equation}\label{eq4}
\det\left( \begin{bmatrix}
\dfrac{1}{p_0(1)}& \cdots& \dfrac{1}{p_n(1)}\\
\vdots&\space& \vdots\\
\dfrac{1}{p_0(n+1)}&\cdots&\dfrac{1}{p_n(n+1)}
\end{bmatrix}\right)=0,
 \end{equation}
and finally by Lemma \ref{lemmacremona} that $ \sigma(P_1),\ldots,\sigma(P_r)  $ are not in general position. 

Conversely, suppose that $ \sigma(P_1),\ldots,\sigma(P_r)  $ are not in general position. So there exists a choice of $ n+1 $ distinct elements of $ \sigma(P_1),\ldots,\sigma(P_r) $ which lie on a hyperplane and without loss of generality we can assume $ \sigma(P_1),\ldots,\sigma(P_{n+1}) $. It implies that \begin{equation}\label{eq5}
\det\left(\begin{bmatrix}
	\sigma(P_1)&\ldots&\sigma(P_{n+1})
\end{bmatrix}^T\right)=0.
\end{equation}
By definition, (\ref{eq4}) follows from (\ref{eq5}) and since $ \lambda_j\neq 0 $ for $ j=1,\ldots,n+1 $, (\ref{eq4}) shows that (\ref{eq3}) holds and so (\ref{eq2}).
We deduce from \ref{eq2} that there exit $ n+1 $ distinct elements in $ \LL $ which are not linearly independent, hence that $ \LL $ is not generally linear.
\end{proof}
\begin{remark} Note that if $\sigma(P_1),\dots,\sigma(P_r)$ are in general position, then the points $\sigma(P_i)$, $\sigma(P_j)$, and $\sigma(P_k)$ are not collinear for all possible choices of $1 \leq i < j < k \leq r$. As in \cite[Theorem 4.3]{CCGV}, there is no rational normal curve containing the coordinates points and the points  $P_i$, $P_j$, and $P_k$.
\end{remark}
\begin{corollary}\label{cor1}
Let $ P_1\ldots,P_r $ be generic points in $ \PP^n $. Let $ H $ be a hyperplane such that $ H\cap \Delta_0=\emptyset $ and set $ H_i=P_i\hada H$ for all $ i\in[r]$. Then $ \mathbb{X}_c(H_1,\ldots,H_r)$ is a WHSC.
\end{corollary}
\begin{proof}
It is enough to use Lemma \ref{genericlemma} and Theorem \ref{teowhsc}.
\end{proof}
\begin{remark}\label{rem6}
Note that being the points $P_1\ldots,P_r $ in general position is necessary but not sufficient to conclude that $ \mathbb{X}_c(H_1,\ldots,H_r)$ is a WHSC. Indeed, from Remark 
\ref{remarkgeneralnogeneric} there exist the points $P_1\ldots,P_r \in \PP^n$ in general position  such that  $ \sigma(P_1),\ldots,\sigma(P_r)  $ are not in general position and so $ \mathbb{X}_c(H_1,\ldots,H_r)$ is not a WHSC.
\end{remark}
\begin{corollary}\label{cor2}
Let $ H\subset\PP^n $ be a hyperplane such that $ H\cap\Delta_0=\emptyset $. Consider $ P_1,\ldots,P_r\in H\setminus \Delta_{n-1} $ and let $ H_i=P_i\hada H$ for all $ i\in[r]$. Then $ \mathbb{X}_c(H_1,\ldots,H_r)$ is a HSC if and only if $  \sigma(P_1),\ldots,\sigma(P_r) $ are in general position in $ \PP^n $.
\end{corollary}
\begin{proof}
From Theorem \ref{teowhsc}, we have that $ \mathbb{X}_c(H_1,\ldots,H_r)$ is a WHSC if and only if $  \sigma(P_1),\ldots,\sigma(P_r) $ are in general position in $ \PP^n $. But $ \mathbb{X}_c(H_1,\ldots,H_r)$ is a HSC too since by hypothesis $P_i\in H$ for all $i\in[r]$.
\end{proof}
\begin{corollary}\label{cor3}
Let $ H\subset\PP^n $ be a hyperplane such that $ H\cap\Delta_0=\emptyset $. Let $ P_1,\ldots,P_r$ be generic points in $H$ and set $ H_i=P_i\hada H$ for all $ i\in[r]$. Then $ \mathbb{X}_c(H_1,\ldots,H_r)$ is a WHSC.
\end{corollary}
\begin{proof}
The proof follows from Lemma \ref{genericlemmahyp} and Corollary \ref{cor2}.
\end{proof}
\begin{remark}
As in Remark \ref{rem6}, being the points $P_1,\ldots,P_r $ in general position in $ H $ is not sufficient to conclude that $ \mathbb{X}_c(H_1,\ldots,H_r)$  is a HSC. To be more precise, by Remark \ref{examplegeneralnogeneric}, there exist the points $P_1,\ldots,P_r $ in general position in $H$ such that $ \sigma(P_1),\ldots,\sigma(P_r)  $ are not in general position in $ \PP^n $, and so $ \mathbb{X}_c(H_1,\ldots,H_r)$  is not a HSC.
\end{remark}
\begin{definition}
If $ \mathbb{X} $ is a finite set of points in $ \PP^n$, then the $ r $-th \emph{square-free} Hadamard product of $ \mathbb{X} $ is 
$$\mathbb{X}^{\underline{\hada}r}=\{P_1\hada\cdots\hada P_r| P_i\in\mathbb{X}\ \text{and}\ P_i\neq P_j\}.$$
\end{definition}
\begin{theorem}\label{teoxweak}
Let $\ell$ be a line in $\PP^n$ such that $\ell\cap\Delta_{n-2}=\emptyset$, and let $\mathbb{X}\subseteq\ell$ be a set of $m>n$ points with $\mathbb{X}\cap\Delta_{n-1}=\emptyset$. Then $\mathbb{X}^{\underline{\hada}n}$ is a WHSC.
\end{theorem}
\begin{proof}
From \cite[Theorem 4.7]{BCK}, we have that $\mathbb{X}^{\underline{\hada}n}$ is a star configuration defined by the set of hyperplanes $\{P\hada\ell^{\hada(n-1)}|P\in\mathbb{X}\}$. The proof is completed by the definition of WHSC.
\end{proof}
In the following, we extend \cite[Theorem 2.17]{CCGV}.
\begin{theorem}\label{thm411}
Let $\ell$ be a line in $\PP^n$ such that $\ell\cap\Delta_{n-2}=\emptyset$, and let $\mathbb{X}\subseteq\ell$ be a set of $m>n$ points such that  $\mathbb{X}\cap\Delta_{n-1}=\emptyset.$
If there exist two distinct points $P=[p_0:\cdots:p_n]$ and $Q=[q_0:\cdots:q_n]$ on $\ell$ such that
\begin{equation}\label{twodet}
\det \left( \begin{bmatrix}
p_0^{n-1}& \cdots&p_n^{n-1}\\
p_0^{n-2}q_0&\cdots &p_n^{n-2}q_n\\
\vdots&\space& \vdots\\
q_0^{n-1}& \cdots&q_n^{n-1}\\
p_0& \cdots&p_n\\
\end{bmatrix}\right) =\det\left( \begin{bmatrix}
p_0^{n-1}& \cdots&p_n^{n-1}\\
p_0^{n-2}q_0&\cdots &p_n^{n-2}q_n\\
\vdots&\space& \vdots\\
q_0^{n-1}& \cdots&q_n^{n-1}\\
q_0& \cdots&q_n\\
\end{bmatrix}\right) =0,
\end{equation}
then $\mathbb{X}^{\underline{\hada}n}$ is a HSC.
\end{theorem}
\begin{proof}
By \cite[Corollary 3.7]{BCK}, $\ell^{\hada(n-1)}$ is given by the following equation:
$$\det\left(  \begin{bmatrix}
p_0^{n-1}& p_1^{n-1}& \cdots&p_n^{n-1}\\
p_0^{n-2}q_0&p_1^{n-2}q_1&\cdots &p_n^{n-2}q_n\\
\vdots&\vdots&\space& \vdots\\
q_0^{n-1}& q_1^{n-1}& \cdots&q_n^{n-1}\\
x_0& x_1& \cdots&x_n\\
\end{bmatrix}\right) =0.$$
By the hypothesis on $P$ and $Q$, we deduce that $P$ and $Q$ are in $\ell^{\hada(n-1)}$, hence that $\mathbb{X}\subseteq\ell\subseteq\ell^{\hada(n-1)}$. From Theorem \ref{teoxweak}, $\mathbb{X}^{\underline{\hada}n}$ is a WHSC and is given by $\{P\hada\ell^{\hada(n-1)}|P\in\mathbb{X}\}$, and thus $\mathbb{X}^{\underline{\hada}n}$ is a HSC by the definition Hadamard star configuration.
\end{proof}

\begin{remark}(\ref{twodet}) is a numerical sufficient condition whether  $\mathbb{X}^{\underline{\hada}n}$ is a HSC. More geometrically, (\ref{twodet}) follows that $P$ and $Q$ are in the linear subspace generated by $P^{{\hada}(n-1)},P^{{\hada}(n-2)}\hada Q,\ldots,P\hada Q^{{\hada}(n-2)},Q^{{\hada}(n-1)}$. Moreover, one can check that if $[1:\cdots:1]\in\ell$, then (\ref{twodet}) holds; so also Theorem \ref{thm411} for all $Q\in \ell$. Furthermore, Theorem \ref{thm411} always holds for $n=2$ (see \cite[Theorem 2.17]{CCGV}).
\end{remark}
\section{Apolar Hadamard star configuration}\label{AHS}
In this section, we study the existence of a (weak) Hadamard star configuration apolar to homogeneous polynomials. In \cite{BC} the authors described the 3-ples $ (d,r,n) $  for which there exists a star configuration $ \mathbb{X}(r):=\mathbb{X}(L_1,\ldots,L_r)$ of codimension $ n $ apolar to the generic $ F\in S_d $. We recall some basic facts.

Let us consider the standard graded polynomial ring $T=\mathbb{C}[y_0,\ldots,y_n]$. We make $ T $ act on $S$ via differentiation, e.g., we think of $ y_j =\partial/\partial x_j $. For any form $ F $ of degree $ d $ in $ S $, we define the ideal $ F^\perp\subseteq T $ say \emph{perp ideal} of $ F $ as follows:
$$F^\perp=\{\partial\in T:\partial F=0\}. $$
\begin{lemma}[Apolarity Lemma]
	A homogeneous degree d form $F\in S$ can be written as
	$$F=\displaystyle\sum_{i=1}^s \alpha_i{L}^d_i,\ {L}_i\in S_1\  pairwise\ linearly\ independent,\ \alpha_i\in\mathbb{C} $$
	if and only if there exists $ I\subseteq F^\perp $ such that $ I $ is the ideal of a set of $ s $ distinct points in $ \mathbb{P}(S_1) $.
\end{lemma}
\begin{definition}
We say that a set of points $ \mathbb{X} $ is {\em apolar to a form $ F $} if the ideal of the set of points is such that  $I({\mathbb{X}})\subset F^\perp $. We say that $ \mathbb{X} $ is an {\em apolar Hadamard star configuration (aHSC) for} $F$ if the set $\mathbb{X} $ is a HSC.
\end{definition}
\begin{remark}\label{rem5.2}
Any linear form $ L\in S_1 $ can be seen as the point $[L]\in\PP(S_1)$, so the set $\LL=\{L_1,\ldots,L_r\}\subset S_1$ of generally linear forms is a set of points in general position in $\PP(S_1)$. 
\end{remark}
\begin{remark}\label{notexist}
Let $F$ be a generic form of degree $d\geq 2$ in $n+1$ variables. If $r<d+n$, then there is Neither WHSC nor HSC apolar to $F$ unless,
$$(d,r,n)=(3,5,3),(4,6,3),(5,7,3),(3,6,4),(3,7,5),\ \text{and}\ (d,d+1,2),$$
 (see \cite[Lemma 3.2, Theorem 3.3, 3.4,  Proposition 3.5 and Conjecture 3.6]{BC}).
\end{remark}
\begin{lemma}\label{proapolar}
Assume that Corollary \ref{cor3} holds. Let $ F $ be a form of degree $ d\geq 2  $ in $n+1$ variables. If $ r\geq d+n $, then there exists an HSC  $\mathbb{X}(H_1,\ldots,H_r) $ apolar to $ F $.
\end{lemma}
\begin{proof}
The desired result follows from \cite[Lemma 3.1]{BC}.   
\end{proof}
\begin{example}\label{ex5.6}
Let $ F=\frac{1}{5}x_0^{2}+x_0 x_1+3 x_1^{2}+\frac{7}{9} x_0 z_2+\frac{5}{4} x_1 x_2+\frac{5}{4} x_2^{2} $ be a generic ternary quadratic form and $ \LL=\{L_1,L_1,L_3,L_4\} $ be a set of generally linear forms, where $L_1=(13/4)y_0+(1/2)y_1+(1/3)y_2$, $ L_2= -(13/15)y_0+(1/3)y_1+(1/6)y_2$, $ L_3= (1/7)y_0+(1/7)y_1+(1/5)y_2$ and $ L_4= y_0+(1/3)y_1+(1/4)y_2$. By Proposition \ref{propsamesupport},  $ \mathbb{X}(\LL) $ is a WHSC, and thus Lemma \ref{proapolar} shows that $ \mathbb{X}(\LL) $ is apolar to $ F $. 
Using \cite[Theorem 3.1]{CCGV}, we conclude that $ \mathbb{X}(\LL) $ is aHSC too since
$$\operatorname{rk}
\left( \begin{bmatrix}
4/13& 2&3\\
{-15/13}&3&6\\
7&7&5\\
1&3&4
\end{bmatrix}\right) =2.
$$
\end{example}
\begin{remark}
Note that, four linear forms $\Gamma_1=y_0+3y_1-2y_2$, $\Gamma_2=-3y_0+5y_1+y_2$, $\Gamma_3=-(1/2)y_0+(1/4)y_1+7y_2$ and $\Gamma_4=4y_0+3y_1+y_2$ construct a WHSC $ \mathbb{X}(\Gamma_1,\ldots,\Gamma_4) $ apolar to $ F $, but do not construct an aHSC since
$$\operatorname{rk}\left( \begin{bmatrix}
1&1/3&{-1/2}\\
{-1/3}&1/5&1\\
{-2}&4&1/7\\
1/4&1/3&1
\end{bmatrix}\right) \neq 2.$$
\end{remark}
\begin{example}\label{nicex}
Let $ M=x_0x_1x_2 $ be a ternary monomial. Since $ M $ has rank 4, the interesting case to check is $r=4$. For $r\neq4$, see Remark \ref{notexist} and Lemma \ref{proapolar}. Let $\mathbb{X}(4)$ be a generic star configuration constructed by linear forms,
\begin{flalign*}
L_1 = a_1y_0 + b_1y_1 + c_1y_2,\  L_2 = a_2y_0 + b_2y_1 + c_2y_2,\\
L_3 = a_3y_0 + b_3y_1 + c_3y_2,\  L_4 = a_4y_0 + b_4y_1 + c_4y_2,
\end{flalign*}
with all $a_i, b_i, c_i$ different from zero. There is no loss of generality in assuming $c_1=c_2=c_3=c_4=1$. An easy calculation follows that $M^\perp=(y_0^2,y_1^2,y_2^2).$
By Apolarity Lemma, the set $\mathbb{X}({L_1,\ldots,L_4})$ is apolar to $M$ if and only if  $I({\mathbb{X}({L_1,\ldots,L_4})})\subset M^\perp$, and it follows that 
\begin{equation}\label{form1}
\begin{array}{c}
b_3a_4 + b_2a_4 + a_3b_4 + a_2b_4 + b_2a_3 + a_2b_3 = 0,\\
b_3a_4 + b_1a_4 + a_3b_4 + a_1b_4 + b_1a_3 + a_1b_3 = 0,\\
b_2a_4 + b_1a_4 + a_2b_4 + a_1b_4 + b_1a_2 + a_1b_2 = 0,\\
b_2a_3 + b_1a_3 + a_2b_3 + a_1b_3 + b_1a_2 + a_1b_2 = 0.
\end{array}
\end{equation}
Assume that (\ref{form1}) has at least one solution with all $ a_i\neq 0 $ and $ b_i\neq0 $ such that $ L_1,\ldots,L_4 $ are generally linear. Thus all maximal minors of the coefficients matrix of $ L_1,\ldots,L_4 $ are non-zeros, i.e.,
\begin{equation}\label{formnew}
\begin{array}{c}
-a_2b_1+a_3b_1+a_1b_2-a_3b_2-a_1b_3+a_2b_3 \not = 0,\\
-a_2b_1+a_4b_1+a_1b_2-a_4b_2-a_1b_4+a_2b_4 \not =  0,\\
-a_3b_1+a_4b_1+a_1b_3-a_4b_3-a_1b_4+a_3b_4 \not =  0,\\
-a_3b_2+a_4b_2+a_2b_3-a_4b_3-a_2b_4+a_3b_4 \not =  0.
\end{array}
\end{equation}
If (\ref{form1}) and (\ref{formnew}) hold, then the set $\mathbb{X}({L_1,\ldots,L_4})$ is a WHSC which is apolar to $ M $. Furthermore, by Corollary \ref{cor2}, it is an aHSC if and only if 
$$\operatorname{rk}\left( \begin{bmatrix}
\frac{1}{a_1} & \frac{1}{b_1} & 1\\
\frac{1}{a_2} & \frac{1}{b_2} & 1\\
\frac{1}{a_3} & \frac{1}{b_3} & 1\\
\frac{1}{a_4} & \frac{1}{b_4} & 1
\end{bmatrix}\right) =2.
$$
Therefore, all maximal minors of the matrix above should be zero,
\begin{equation}\label{form2}
\begin{array}{c}
\frac{1}{a_1b_2}+\frac{1}{a_3b_1}+\frac{1}{a_2b_3}-\frac{1}{a_3b_2}-\frac{1}{a_2b_1}-\frac{1}{a_1b_3}=0,\\
\frac{1}{a_1b_2}+\frac{1}{a_4b_1}+\frac{1}{a_2b_4}-\frac{1}{a_4b_2}-\frac{1}{a_2b_1}-\frac{1}{a_1b_4}=0,\\
\frac{1}{a_1b_3}+\frac{1}{a_4b_1}+\frac{1}{a_3b_4}-\frac{1}{a_4b_3}-\frac{1}{a_3b_1}-\frac{1}{a_1b_4}=0,\\
\frac{1}{a_2b_3}+\frac{1}{a_4b_2}+\frac{1}{a_3b_4}-\frac{1}{a_4b_3}-\frac{1}{a_3b_2}-\frac{1}{a_2b_4}=0.
\end{array}
\end{equation}
Since all $a_i$ and $b_i$ are non-zero, (\ref{form2}) is equivalent: 
\begin{equation}\label{form3}
\begin{array}{c}
a_1a_2a_3b_1b_2b_3\left(\frac{1}{a_1b_2}+\frac{1}{a_3b_1}+\frac{1}{a_2b_3}-\frac{1}{a_3b_2}-\frac{1}{a_2b_1}-\frac{1}{a_1b_3}\right)=0\\
a_1a_2a_4b_1b_2b_4\left(\frac{1}{a_1b_2}+\frac{1}{a_4b_1}+\frac{1}{a_2b_4}-\frac{1}{a_4b_2}-\frac{1}{a_2b_1}-\frac{1}{a_1b_4}\right)=0\\
a_1a_3a_4b_1b_3b_4\left(\frac{1}{a_1b_3}+\frac{1}{a_4b_1}+\frac{1}{a_3b_4}-\frac{1}{a_4b_3}-\frac{1}{a_3b_1}-\frac{1}{a_1b_4}\right)=0\\
a_2a_3a_4b_2b_3b_4\left(\frac{1}{a_2b_3}+\frac{1}{a_4b_2}+\frac{1}{a_3b_4}-\frac{1}{a_4b_3}-\frac{1}{a_3b_2}-\frac{1}{a_2b_4}\right)=0.\\
\end{array}
\end{equation}
Let $ I, J $, and $ K $ be the ideals generated by the equations of (\ref{form1}), (\ref{formnew}), and (\ref{form3}), respectively. We conclude that any point in the zero locus of $ I'=(I:J^\infty)\subset\mathbb{C}[a_1,\dots,a_4,b_1,\dots,b_4] $ is equivalent to a WHSC apolar to $M$, unless some $a_i$ and $b_i$ are zeros. Moreover, any point $[a_1:\cdots:b_4]\in V(I'+K)  $ with all $a_i\neq0$ and $b_i\neq0$ constructs an aHSC for $ M $. A simple computation in \verb|Macaulay2|\cite{GS} shows that $ K\subset I' $ and so $ I'+K=I' $, thus we conclude that a WHSC apolar to $M$ is an aHSC too. Using \verb|Macaulay2| and \verb|Bertini_real|\cite{B_R} helped us to find a point on $  V(I') $, that is,
$$
a_1=\frac{\sqrt{2641}+119}{4(\sqrt{2641}+47)},\ a_2=\frac{2(-\sqrt{2641}-59)}{\sqrt{2641}+47},\ a_3=2,\ a_4=1,
$$
$$b_1=\frac{3(-\sqrt{2641}-39)}{16},\ 
b_2=5, \ b_3=9, \ b_4=\frac{\sqrt{2641}+11}{4}.
$$
Therefore, it verifies the existence of an aHSC $ \mathbb{X}(4) $ for $ M $.
\end{example}
\begin{remark}
Note that from Example \ref{nicex} it follows that if we have a star configuration $ \mathbb{X}(L_1,\dots,L_4) $ apolar to $ M $ such that $ V(L_i)\cap \Delta_0=\emptyset $, then there exists a line $ L $ and four points $ P_1,\dots,P_4\in V(L) $ such that $ V(L_i)=V(L)\hada P_i $ for all $ i $.
\end{remark}
\begin{remark}\label{remhold}
A generic form $ F $ has an apolar WHSC if and only if there exists a star configuration $ \mathbb{X}(L_1,\ldots,L_r) $ apolar to $ F $ such that $V(L_i)\cap\Delta_0=\emptyset$ for all $i\in[r]$.
\end{remark}
\begin{remark}
The authors in \cite[Conjecture 3.6]{BC} suggest that any generic ternary form of degree $ d\geq 3$ has an apolar star configuration $ \mathbb{X}(d+1)$. One can see that the conjecture is also satisfied for the existence of an apolar WHSC $ \mathbb{X}(d+1)$ if Remark \ref{remhold} holds.
\end{remark}

\begin{corollary}
There exists an apolar WHSC $ \mathbb{X}(4)$ for any ternary cubic of rank five.
\end{corollary}
\begin{proof}
By \cite[Proposition 4.4]{BC} and Proposition \ref{propsamesupport}, the proof is done.
\end{proof}


\begin{thebibliography}{}
\bibitem[1]{BC}
I. Bahmani Jafarloo,  E. Carlini (2019) \textit{Special apolar subset: the case of star configurations}, Communications in Algebra, DOI: \url{https://doi.org/10.1080/00927872.2019.1677684}

\bibitem[2]{BCK}
C. Bocci,  E. Carlini  and J. Kileel. {\it  Hadamard products of linear spaces}, Journal of Algebra {\bf 448} (2016) 595-617.

\bibitem[3]{BCFL1}
C. Bocci, G. Calussi, G. Fatabbi, A. Lorenzini. {\it On Hadamard products of linear varieties}, J. Algebra and Appl {\bf 16} (2017) art. no. 1750155.

\bibitem[4]{BCFL2}
C. Bocci, G. Calussi, G. Fatabbi, A. Lorenzini. {\it The Hilbert function of some Hadamard products}, Collect. Math. (2017) DOI: \url{https://doi.org/10.1007/s13348-017-0200-z}.

\bibitem[5]{B_R}
D.A. Brake, D.J. Bates, W. Hao, J.D. Hauenstein, A.J. Sommese and C.W. Wampler. Bertini\_real: \textit{software for real algebraic sets}. Available online at \url{bertinireal.com}. 

\bibitem[6]{CCFL} G. Calussi, E. Carlini, G. Fatabbi, A. Lorenzini. {\it Hadamard products of degenerate subvarieties}, arXiv:1804.01388.

\bibitem[7]{CCGV} E. Carlini, M.V. Catalisano, E. Guardo, A. Van Tuyl.{\it Hadamard star configurations}. Rocky Mountain J. Math. \textbf{49} (2019), no. 2, 419--432.


\bibitem[8]{CGV} E. Carlini, E. Guardo, A. Van Tuyl. {\it Star configurations on generic hypersurfaces}, J. Algebra {\bf 407} (2014), 1--20.

\bibitem[9]{CMS}  M.A. Cueto, J. Morton and B. Sturmfels. {\it Geometry of the restricted Boltzmann machine}, In: M. Viana and H. Wynn (eds) {\it Algebraic Methods in Statistics and Probability}, American Mathematicals Society, Contemporary Mathematics {\bf 516} (2010) 135--153.

\bibitem[10]{CTY} M.A. Cueto, E.A. Tobis and J. Yu. {\it An implicitization challenge for binary factor analysis}, J. Symbolic Comput.,  \textbf{45} (2010), no. 12, 1296--1315.

\bibitem[11]{FOW} N. Friedenberg, A. Oneto, R.L. Williams. {\it Minkowski sums and Hadamard product of algebraic vatieties}, In: G. Smith, B. Sturmfels (eds) {\it Combinatorial Algebraic Geometry}. Fields Institute Communications, {\bf 80}. Springer (2017) 133-157.

\bibitem[12]{GHM} A.V. Geramita, B. Harbourne, and J. Migliore. {\it Star configurations in $ \PP^n $}, J. Algebra {\bf 376} (2013), 279--299.

\bibitem[13]{GS} Daniel R. Grayson and Michael E. Stillman. Macaulay2, a software system for research in algebraic geometry. Available at \url{https://faculty.math.illinois.edu/Macaulay2/}.

\bibitem[14]{MS} D. Maclagan, B. Sturmfels. {\it Introduction to tropical geometry}. Graduate Studies in Mathematics, American Mathematicals Society, {\bf 161} (2015).
\end{thebibliography}
\end{document}